\documentclass{amsart} 
\usepackage{enumerate, amssymb, amsfonts, latexsym}
\usepackage{amscd, amsmath}
\usepackage{graphicx}
\newtheorem{theorem}{Theorem} 
\newtheorem{lemma}{Lemma} 
\newtheorem{definition}{Definition} 
 
\newtheorem{prop}{Proposition}

\begin{document} 

\title{The Calabi Conjecture}
\author{Rohit Jain and Jason Jo}

\begin{abstract}

In this paper we aim to explore the Geometric aspects of the Calabi Conjecture
and highlight the techniques of nonlinear Elliptic PDE theory used by S.T. Yau
[SY] in obtaining a solution to the problem. Yau proves the existence of a
Geometric
structure using differential equations, giving importance to the idea that deep
insights into geometry can be obtained by studying solutions of such equations.
Yau's proof of the existence of a specific class of metrics have found a natural
interpretation in recent developments in Theoretical Physics most notably in the
formulation of String Theory. We will also attempt to explore the importance of
a special case of Yau's solution known as Calabi-Yau Manifolds in the context of
holonomy. 

\end{abstract}

\maketitle

\tableofcontents

\section{K\"{a}hler Geometry}
There are many mathematical reasons to study K\"{a}hler geometry. However, we
wish to give a physical motivation to study K\"{a}hler manifolds. In string
theory, the universe is conjectured to be 10 dimensional (9 space dimensions, 1
time dimension) with 6 of the space dimensions compactified into a Planck
scale manifold, let it be denoted as $\mathcal{M}$. During the early days of
string theory, it was believed that $\mathcal{M}$ was a flat hypertorus, i.e.
$\mathcal{M} = S^1 \times S^1 \times S^1 \times S^1 \times S^1 \times S^1$.  But
the hypertorus based string theory could not incorporate chiral
aspects of the Standard Model. Thus, in order for String theory to be a unified
field theory, there must be a more sophisticated choice for $\mathcal{M}$. It
turns out that the correct choice for $\mathcal{M}$ are so called Calabi Yau
manifolds [POL] which are examples of K\"{a}hler manifolds.
K\"{a}hler geometry also naturally occurs in the context of supersymmetry and
$\sigma$-models. So in addition to its inherent mathematical interest,
K\"{a}hler
geometry is also of fundamental importance to high energy physics research. For
the remainder of this section, we will give an introduction to some of the
K\"{a}hler geometry concepts necessary for the Calabi Conjecture.

\begin{definition} A complex manifold of complex dimension $n$ (hence real
dimension $2n$) is a smooth manifold equipped with an atlas of charts in which
all the transition functions are holomorphic functions.
\end{definition}

Complex manifolds come equipped with a \textit{complex structure}, denoted by
$J$. In particular, we have the following definition:
\begin{definition} An almost complex manifold $M$ is a (real) smooth manifold
with a globally defined $(1,1)$ tensor $J$ which is an endomorphism of the
tangent bundle $TM$ such that:
\begin{equation} J^2 = - \mathbb{I},
\end{equation}
$J$ is called an almost complex structure.
\end{definition}

Now let us consider the situation locally. For any given point $p \in M$, we
have an endomorphism $J_p: T_p M \rightarrow T_p M$ which satisfies $J_p^2 = -
\mathbb{I}_p$ depending smoothly on $p$, where $\mathbb{I}_p$ is the identity
operator on the tangent space $T_p M$. Given a local coordinate system $x^\mu,
\mu = 1, \dots, \dim(M)$, then we locally have the representation:
\begin{equation} J_p = J_\mu{}^\nu(p) \frac{\partial}{\partial x^\nu} \otimes
dx^\mu.
\end{equation}
It should be noted that $J_\mu{}^\nu(p)$ is a real function in the real basis.
Hence, for any vector field $X = X^\mu \frac{\partial}{\partial x^\mu}$, our
endomorphism $J$ acts on $X$ according to:
\begin{align} J(X) &= (X^\mu J_\mu{}^\nu)\frac{\partial}{\partial x^\nu},
\\ \Rightarrow J^2(X) &= X^\sigma J_\sigma{}^\nu J_\nu{}^\mu
\frac{\partial}{\partial x^\mu}.
\end{align}

Thus, in local coordinates the conditions for an almost complex structure is
equivalent to the following matrix equation:
\begin{equation} J_\sigma{}^\nu(p) J_\nu{}^\mu(p) = -\delta_\sigma{}^\mu.
\end{equation}
Globally, the existence of an almost complex structure $J$ on a smooth manifold
means that we can define $J_p$ in any patch and glue them together without any
obstructions or singularities.

Observe that given a complex manifold $M$, we may view it as a complex manifold
of complex dimension $n$ or we may view it as the underlying \textit{real}
manifold of dimension $2n$. If we have complex coordinates $z^j$, for $j = 1,
\dots, n$, we have the corresponding real coordinates $x^j, y^j$ for $j = 1,
\dots, n$, where we use the identification $z^j = x^j+iy^j$. Thus we get a set
of $2n$ coordinates $z^j, z^{\overline{j}}$, where $z^{\overline{j}} \equiv x^j
- i
y^j$ for $a = 1, \dots, n$. We make the following definitions:
\begin{align} \frac{\partial}{\partial z^j} &\equiv \frac{1}{2}
\left(\frac{\partial}{\partial x^j} - i  \frac{\partial}{\partial y^j}\right),
\\ \frac{\partial}{\partial z^{\overline{j}}} &\equiv \frac{1}{2}
\left(\frac{\partial}{\partial x^j} +i  \frac{\partial}{\partial y^j} \right),
\\ dz_j &\equiv dx_j + i dy_j
\\ dz_{\overline{j}} &\equiv dx_j - i dy_j.
\end{align}

We then have the following differential operators:
\begin{align} \partial &= \frac{1}{2} \left(  \frac{\partial}{\partial x^j} -
i \frac{\partial}{\partial y^j}  \right) (dx^j + i dy^j),
\\ \overline{\partial} &=  \frac{1}{2} \left(  \frac{\partial}{\partial x^j} +
i \frac{\partial}{\partial y^j}  \right) (dx^j - i dy^j).
\end{align}

With these definitions in hand, we have the following proposition on our
$\partial, \overline{\partial}$ differential operators:
\begin{prop} The exterior derivative $d$ satisfies:
\begin{equation} d = \partial + \overline{\partial}.
\end{equation}
Furthermore,
\begin{align} \partial \partial &= 0, \overline{\partial}  \overline{\partial}
= 0,
\\ \partial \overline{\partial}  &= - \overline{\partial}  \partial.
\end{align}
\end{prop}
\begin{proof} We have the following:
\begin{align} \partial + \overline{\partial} &= \frac{1}{2} \left(
\frac{\partial}{\partial x^j} - i \frac{\partial}{\partial y^j} \right) (dx^j +
i dy^j) + \frac{1}{2}\left(   \frac{\partial}{\partial x^j} + i
\frac{\partial}{\partial y^j} \right)(dx^j - i dy^j)
\\&= \frac{\partial}{\partial x^j}dx^j + \frac{\partial}{\partial y^j}dy^j = d.
\end{align}

Therefore,
\begin{equation} 0 = d^2 = (\partial + \overline{\partial}) (\partial +
\overline{\partial})  = \partial^2 + \partial \overline{\partial} +
\overline{\partial} \partial + \overline{\partial}^2
\end{equation}

Decomposing the above equation into types yields that each piece $\partial^2,
{\overline{\partial}}^2, \partial \overline{\partial} + \overline{\partial}
\partial$ vanishes. This proves the theorem.
\end{proof}

With these $\partial, \overline{\partial}$ differential operators, we may
decompose $\Omega^k$ the space of $k$-forms into subspaces $\mathcal{A}^{p,q}$ where
$p + q = k$. Namely, $\mathcal{A}^{p,q}$ is locally spanned by forms of the type
\begin{equation} \omega(z) = \eta(z) dz^{j_1} \wedge \dots \wedge dz^{j_p}
\wedge d\overline{z}^{k_1} \wedge \dots \wedge d\overline{z}^{k_q}
\end{equation}

The almost complex structure $J$ allows us to complexify our real $2n$
dimensional tangent space, let $(T_p M)^\mathbb{C}$ denote the complexified
tangent space. We reinterpret $J$ as a complex linear map still squaring to
$-\mathbb{I}$ at each point $p \in M$. Due to the fact that it squares to
$-\mathbb{I}$, the eigenvalues of $J_p$ can only be $\pm i$ and this allows us
to decompose our tangent space $(T_p M)^\mathbb{C}$:
\begin{equation} (T_p M)^\mathbb{C} = T_p M^+ \oplus T_p M^-,
\end{equation}
where $T_p M^+$ is the eigenspace of $i$ and $T_p M^-$ is the eigenspace of
$-i$.
We note that the complex structure $J$ is an endomorphism of the tangent space
given in local coordinates by the following two conditions:
\begin{align} J(\partial / \partial x_j) &= \partial / \partial y_j
\\ J(\partial / \partial y_j) &= - \partial / \partial x_j,
\end{align}
so clearly $J^2 = - \mathbb{I}$.

A natural question to now ask is what is the relationship between complex and
almost complex manifold?
\begin{theorem} Any complex manifold $M$ is also an almost complex manifold.
\end{theorem}
\begin{proof} Recall that a complex manifold admits a holomorphic atlas, giving
us a complex coordinate system $z^a$ in a neighborhood $U$ of an arbitrary point
$p \in M$. Thus, we may define the tensor:
\begin{equation} J = i \frac{\partial}{\partial z^j} \otimes dz^j - i
\frac{\partial}{\partial \overline{z}^j} \otimes d\overline{z}^j.
\end{equation}
It is important to note that the above equation is well defined in a patch $U$
as opposed to only a point (which would be the case for an almost complex
manifold) since we may define complex coordinates that vary holomorphically in
the patch $U$ for complex manifolds, but we cannot do this for almost complex
manifolds. Thus, we have an almost complex structure defined in any given patch.
Next, we need to check if it is well defined in the overlaps of patches $(U, z)
\cap (V, w)$.

As the coordinate transformation functions $z^j(w)$ are holomorphic, it follows
from basic transformations of vectors and one forms that we have the following:
\begin{equation} \frac{\partial}{\partial z^j} \otimes dz^j = \frac{\partial
z^j}{\partial w^l} \frac{\partial w^k}{\partial z^j} \frac{\partial}{\partial
w^k}\otimes dw^l = \frac{\partial}{\partial w^j}\otimes dw^j.
\end{equation}
Thus we may conclude that in the overlap, $J$ takes the form:
\begin{equation} J = i \frac{\partial}{\partial w^j} \otimes dw^j - i
\frac{\partial}{\partial \overline{w}^j} \otimes d\overline{w}^j,
\end{equation}
which proves the theorem.\end{proof}

\begin{definition} Let $M$ be a complex manifold with Riemannian metric $g$ and
complex structure $J$. If $g$ satisfies:
\begin{equation} g(JX, JY) = g(X,Y) \textrm{ i.e. a compatability condition},
\end{equation}
for any two sections $X,Y$ of the tangent bundle, then $g$ is said to be a
Hermitian
metric and the pair $(M, g)$ is called a Hermitian manifold.
\end{definition}

\begin{prop} A complex manifold $(M, J)$ always admits a Hermitian metric.
\end{prop}
\begin{proof} From rudimentary Riemannian geometry, we know that any manifold
admits a
Riemannian metric $g$ (we always locally have one, and we patch these together
via a partition of unity). To obtain a Hermitian metric, simply define:
\begin{equation} h(X,Y) \equiv \frac{1}{2} ( g(X,Y) + g(JX, JY)).
\end{equation}
Observe that
\begin{align}h(JX, JY) &= \frac{1}{2} ( g(JX, JY) + g(J^2 X, J^2 Y))
\\ &= \frac{1}{2} (g(JX, JY) + g(-X,-Y))
\\ &= \frac{1}{2} ( g(JX, JY) + g(X,Y)) = h(X,Y).
\end{align}
Thus $h$ is indeed a Hermitian metric. \end{proof}

\begin{definition}Given the data $(M, g, J)$, where $M$ is a complex manifold,
$g$ is a Riemannian metric and $J$ is a compatible complex structure, we may
associate a 2 form (in particular a real (1,1) form) denoted $\omega$, defined
by:
\begin{equation} \omega(v, w) = g(Jv, w).
\end{equation}
\end{definition}
We are well within our rights to wonder why is $\omega$ a 2 form? Observe that:
\begin{equation} \omega(X,Y) = g(JX,Y) = g(J^2 X, JY) = -g(X, JY) = -g(JY, X) =
- \omega(Y,X).
\end{equation}
So we see that $\omega$ is indeed a antisymmetric 2 tensor, i.e. it is a 2-form.

In local coordinates, we have the following:
\begin{align} g &= \frac{1}{2} \sum_{i,j} h_{ij}(z, \overline{z}) \left( dz_i
\otimes d \overline{z}_j + d\overline{z}_j \otimes dz_i    \right),
\\ \omega &= \frac{i}{2} \sum_{i,j} h_{ij}(z, \overline{z})  dz_i \wedge
d\overline{z}_j,
\end{align}
where $ h_{ij}$ is positive definite and Hermitian matrix.

\begin{definition} Let $(M, g, J)$ be a Hermitian manifold with associated 2
form $\omega$. If $\omega$ is closed, i.e. $d \omega = 0$, then $M$ is called a
K\"{a}hler manifold, $g$ is the K\"{a}hler metric and $\omega$ the K\"{a}hler
form.
\end{definition}

It goes without saying that there are some marvelous properties of K\"{a}hler
manifolds. We will now quote one of the nicer properties of K\"{a}hler manifolds:
\begin{prop} (The Global $\partial \overline{\partial}$ Lemma): If $M$ is a compact, K\"{a}hler manifold and $\alpha \in \mathcal{A}^{p,q}(M)$ and $d\alpha = 0$, then the following are equivalent:
\begin{itemize}
\item $\alpha$ is $d$-exact
\item $\alpha$ is $\partial$-exact
\item $\alpha$ is $\overline{\partial}$-exact
\item $\alpha$ is $\partial \overline{\partial}$-exact.
\end{itemize}
\end{prop}

The Global $\partial \overline{\partial}$ Lemma allows us to characterize
``cohomologous" K\"{a}hler metrics $g$ and $\widetilde{g}$:
\begin{prop} For $M$ a compact, K\"{a}hler manifold and $g, \widetilde{g}$ two
K\"{a}hler metrics and $\omega, \widetilde{\omega}$ their respective K\"{a}hler
forms. Suppose that $\omega, \widetilde{\omega}$ are cohomologous, i.e.
$[\omega] = [\widetilde{\omega}] \in H^2(M, \mathbb{R})$. Then there exists a
smooth, real function $\varphi$ on $M$ such that $\widetilde{\omega} = \omega +
i\partial \overline{\partial} \varphi$. Moreover, $\varphi$ is unique up to a constant.
\end{prop}
\begin{proof} Because we have that $\omega, \widetilde{\omega}$ are
cohomologous, we have that $\omega - \widetilde{\omega}$ is $d$-exact, real
(1,1) form. Thus $\omega - \widetilde{\omega}$ is actually $i\partial
\overline{\partial}$-exact by the \textit{Global $\partial \overline{\partial}$
Lemma}. Hence $\omega - \widetilde{\omega} = i\partial \overline{\partial}
\varphi$.

Suppose we have two solutions $\varphi_1, \varphi_2$. Then $\partial
\overline{\partial} (\varphi_1 - \varphi_2) = 0$. Thus again by \textit{Global
$\partial \overline{\partial}$
Lemma}, we have that $d(\varphi_1 - \varphi_2) = 0$, i.e. $\varphi_1 -
\varphi_2$ is constant, so any two solutions differ by a constant.
\end{proof}

Thus we see that if two K\"{a}hler metrics $g, \widetilde{g}$ have cohomologous K\"{a}hler forms,
then the metrics $g$ and $\widetilde{g}$ are related by:
$$ \widetilde{g}_{j \bar{k}} = g_{j \bar{k}} + \frac{\partial^2}{\partial z^j
\partial \bar{z}^k} 
$$

Another consequence of K\"{a}hler manifolds comes from the form $\omega$ being
closed, $d \omega = 0$: the existence of a \textit{K\"{a}hler potential}. The K\"{a}hler potential condition says that for any point $p \in M$ and any local patch $U \ni p$, there exists a smooth, real valued function $\varphi$ such that locally, we have:
 \begin{equation} g_{i \overline{j}} = \frac{\partial^2 \varphi}{\partial z^i \partial \overline{z}^j}.
 \end{equation}
However, this $\varphi$ is only really defined locally. In general, since $d
\omega = 0$, it defines a cohomology class $[\omega] \in H^2(M, \mathbb{R})$,
let it be denoted as the \textit{K\"{a}hler class}. Since the real dimension of
the underlying real manifold is $2n$, we have that $\omega^n$, the $n$-times
wedge product is equal to $n!$ times $\textrm{Vol}_g$:
\begin{equation} \int_M \omega^n = n! \textrm{Vol}(M) > 0.
\end{equation}
As $\int_M \omega^n$ is an invariant of the cohomology class, we conclude that
$\omega^n \neq 0$. On the other hand as $i\partial \overline{\partial} \varphi$
is an exact form, we have that $[i\partial \overline{\partial} \varphi] = [0]$.
Thus, on a compact K\"{a}hler manifold, \textit{it is impossible to find a
global K\"{a}hler potential.}

We also have a necessary and sufficient condition for a Hermitian manifold $(M,
g, J)$ to be a K\"{a}hler manifold in terms of the Riemannian geometry data,
namely the Levi-Civita connection $\nabla$:
\begin{theorem} A Hermitian manifold $(M, g, J)$ is K\"{a}hler if and only if
the complex structure $J$ satisfies the following parallel transport equation:
\begin{equation} \nabla J = 0,
\end{equation}
where $\nabla$ is the Levi-Civita connection of $g$.
\end{theorem}
\begin{proof} $(\Leftarrow)$ Assume that $\nabla J = 0$. Recall that $J$ is a
globally defined $(1,1)$-tensor. Then $\nabla J= 0$ means that the components of
$J$ are covariantly constant, $\nabla_\mu J_\nu{}^\rho = 0$. Furthermore, since
$\nabla$ is the Levi-Civita connection, the components of the metric $g$ are
also covariantly constant. Thus the K\"{a}hler form $\omega$ satisfies $\nabla
\omega = 0$. This clearly implies that $d \omega = 0$. Hence $M$ is K\"{a}hler.

$(\Rightarrow)$ For a proof of the converse, we reference Andy Neitzke's Complex Geometry lecture notes [ANCG].
\end{proof}

Let's now discuss the Riemannian Geometry of K\"{a}hler manifolds. Recall that for any arbitrary Riemannian manifold $(M,g)$, let $\nabla$ denote the Levi-Civita connection for $g$. Given any set of local coordinates $x^\mu$, we have that the Christoffel symbols are given by the following equation:
\begin{equation} \Gamma_{\mu \nu}{}^\rho = \frac{1}{2} g^{\rho \sigma} \left( \partial_\mu g_{\sigma \nu} + \partial_\nu g_{\mu \sigma} - \partial_\sigma g_{\mu \nu}   \right).
\end{equation}
More specifically, in local coordinates, the connection $\nabla$ defines a covariant derivative on tensors. So for $T$ a (1,1) tensor, we have the following:
\begin{equation} \nabla_\mu T_\nu{}^\rho = \partial_\mu T_\nu{}^\rho + \Gamma_{\mu \sigma}{}^\rho T_\nu{}^\sigma - \Gamma_{\mu \nu}{}^\sigma T_\sigma{}^\rho.
\end{equation}

Not surprisingly, for a K\"{a}hler manifold, we have an elegant prescription for our Christoffel symbols:
\begin{lemma} For a K\"{a}hler manifold $(M, g, J)$, we have that in complex
coordinates $z^j, \overline{z}^j$ in a neighborhood of a point $z_0 \in M$, the
only non-vanishing components of the Christoffel symbols are:
\begin{equation} \Gamma_{jk}{}^l = (\partial_j g_{k \overline{m}}) g^{\overline{m} l}, \Gamma_{\overline{jk}}{}^{\overline{l}} = \left( \Gamma_{jk}{}^l \right)^*.
\end{equation}
Moreover:
\begin{equation} \Gamma_{jk}{}^k = \partial_j(\log \sqrt{\det g}).
\end{equation}
\end{lemma}
\begin{proof} The first condition follows two facts. The first fact is that $g$
is locally given by the Hermitian positive definite matrix $h_{ij}$, therefore
$h_{ij} = h_{\overline{ij}}= 0$, i.e. the only nonzero components of our metric
are the components corresponding to mixed holomorphic and antiholomorphic
indices. The second fact is that since $g$ is a K\"{a}hler metric, $d \omega = 0
\Rightarrow \partial_j h_{k \overline{l}} = \partial_k h_{j \overline{l}}$ and
the complex conjugate of these statements. For a proof of this statement, we
refer the reader to [JST]. Thus we have that:
\begin{align}  \Gamma_{jk}{}^k &= (\partial_j g_{k \overline{l}})g^{\overline{l}k}
\\ &= \frac{1}{2} \textrm{Tr}[g^{-1} \partial_j g]
\\&= \frac{1}{2} \partial_j [\log \det g]
\\&= \partial_j \log \sqrt{\det g}.
\end{align}
\end{proof}

A byproduct of all these nice K\"{a}hler identities is that the Ricci tensor takes a particularly nice form. Again, the pure holomorphic and pure antiholomorphic components of the Ricci tensor will vanish. The nonvanishing components of the Ricci tensor will given by:
\begin{equation} R_{j \overline{k}} = \partial_j \partial_{\overline{k}} [ \log \sqrt{\det g}]. 
\end{equation} 
Given this fact, we may define the \textit{Ricci Form} given by $R = i R_{j \overline{k}} dz^j \wedge d\overline{z}^k = i \partial \overline{\partial} \log \sqrt{\det g}$. This Ricci Form $R$ defines a cohomology class, which we call the \textit{first Chern class} $c_1 = [\frac{1}{2\pi} R]$. 

We now have introduced all the relevant concepts that will allow us to begin discussing the Calabi Conjecture. However, before we move on to the statement and philosophy of the Calabi Conjecture, we will actually introduce and prove what will turn out to be a crucial little lemma:
\begin{lemma} Let $(M, g)$ be a K\"{a}hler manifold with K\"{a}hler form $\omega$. For a $f \in C^0(M)$, we may define $A$ via:
\begin{equation} A \int_M e^f dV_g = \textrm{Vol}_g(M).
\end{equation} 
Suppose that $\varphi \in C^2(M)$ satisfies $(\omega + i\partial
\overline{\partial} \varphi)^n = Ae^f \omega^n$ on $M$. Then we have that
$\omega + i\partial \overline{\partial} \varphi$ is a positive (1,1) form. 
\end{lemma} 
\begin{proof} Consider any set of holomorphic coordinates $z^j, j = 1, \dots, n$ on a connected patch $U \subset M$. Then since $\omega$ and $\omega + i\partial \overline{\partial} \varphi$ are cohomologous, we have a new metric $\widetilde{g}$ in the patch $U$ given by:
\begin{equation} \widetilde{g}_{j \overline{k}} = g_{j \overline{k}} + \frac{\partial^2 \varphi}{\partial z^j \partial \overline{z}^k}.
\end{equation} 
Up to the reader's conventions, we have that $\omega^n = n! \ d\textrm{Vol}_g$, thus $\widetilde{g}_{j \overline{k}}$ is a Hermitian metric if and only if $\omega + i\partial \overline{\partial} \varphi$ is a positive (1,1) form, i.e. if and only if the eigenvalues of $\widetilde{g}_{j \overline{k}}$ are all positive. Hence, it suffices to show that $\widetilde{g}_{j \overline{k}}$ has positive eigenvalues. 

To this end, by hypothesis, we have that: $(\omega + i\partial
\overline{\partial} \varphi)^n = Ae^f \omega^n$ on $M$. This immediately implies
that:
\begin{equation} \det(\widetilde{g}) = Ae^f \det(g). 
\end{equation} 
Thus $\det(\widetilde{g}) > 0$ and hence must have no zero eigenvalue. But by continuity, if the eigenvalues of $\widetilde{g}$ are positive at some point $z_0$, they must be positive in a neighborhood of $z_0$. By connectedness, we conclude that the eigenvalues must be positive on all of $M$. 

By compactness of $M$ and continuity of $\varphi$, $\varphi$ obtains a minimum at some point $p \in M$. Let $U$ be a patch about a minimum point $p$. Since $p$ is a point where $\varphi$ achieves it's minimum, we conclude that the matrix $ \left[  \frac{\partial^2 \varphi}{\partial z^j \partial \overline{z}^k} \right] $ has positive eigenvalues. We may conclude this since in higher dimensions, the Hessian of $\varphi$ must be positive definite at a minimum point $p$. As $g$ is a Hermitian metric, it has only positive eigenvalues as well. Thus we conclude that $\widetilde{g}$ must have positive eigenvalues and thus defines a Hermitian metric and hence we may conclude that it's K\"{a}hler form $\omega + i\partial \overline{\partial} \varphi$ is a positive (1,1) form.
\end{proof}

\section{Introduction to Calabi Conjecture}
We start by assuming that $M$ is a compact K\"{a}hler Manifold with K\"{a}hler
metric $g$ = $\sum_{i,j} g_{i\bar{j}} dz^i \otimes dz^{\bar{j}}$ and Ricci
Tensor $R$ = $\sum_{i,j} R_{i\bar{j}} dz^i \otimes dz^{\bar{j}}$. We showed
above that
\\ 
\begin{align} 
R_{i\bar{j}} = - \frac{\partial^2}{\partial z^i \partial z^{\bar{j}}}
[\log\det(g_{k\bar{l}})].
\end{align}
\\
\textit{Note:} Up to sign choices and the constant $\frac{1}{2}$, our Ricci
tensor agrees with Yau's. Note in Equation 32, we have a leading $\frac{1}{2}$
factor, while Yau does not and the minus sign comes from choice of $\sqrt{-1}$
convention. 
\\
\\
This implies that  the closed real (1,1) form $\frac{i}{2\pi}\sum_{i,j}
R_{i\bar{j}} dz^i \wedge dz^{\bar{j}}$ can equivalently be written as
$-\frac{i}{2\pi}\partial\bar{\partial}[\log\det(g_{k\bar{l}})]$. According
to a theorem proved by S.S. Chern [SS] the cohomology class of this particular
(1,1) form depends only on the Complex Structure of $M$. Furthermore this
closed, real (1,1) form is exactly equal to the first Chern Class of $M$. The
Calabi Conjecture is to look at the converse of this statement. 
\begin{theorem}
(Calabi Conjecture 1954) Let M be a compact, complex manifold and g a K\"{a}hler
metric on M with K\"{a}hler form $\omega$. Suppose $\tilde{R}$ is a closed, real
(1,1) form on M with $[\tilde{R}]$ = $c_{1}(M)$. Then there exists unique
K\"{a}hler metric $\tilde{g}$ on M with K\"{a}hler form $\tilde{\omega}$ such
that $[\tilde{\omega}]$ = $[\omega]$ $\in H^{2}(M,\mathbb{R})$ and the Ricci
form of $\tilde{\omega}$ is  $\tilde{R}$.   
\end{theorem}
Calabi was able to show that if such a  $\tilde{g}$ exists then it must be
unique. \textit{The key insight to solving the problem was to reduce the problem
to a nonlinear Elliptic PDE of Monge-Amp\`{e}re type.}
\\  

Assume that the Calabi Conjecture is in fact true. Since the metrics have been
shown to be cohomologous, we know by the Global $\partial\bar{\partial}$-Lemma,
that $\exists$ a smooth real function $\varphi$ unique up to the addition of a
constant, such that $\tilde{\omega} = \omega + i\partial\bar{\partial}\varphi$.
Furthermore locally we have the following representation for our metrics:
$\tilde{g}_{i\bar{j}} =  g_{i\bar{j}}+\frac{\partial^2 \varphi}{\partial z^i
\partial z^{\bar{j}}}$. Note that this allows us to paramaterize a class of
K\"{a}hler metrics by a single smooth function $\varphi$.
\\

Turning our attention to their respective Ricci forms let $\tilde{R} =
\frac{i}{2\pi}\sum_{i,j} \tilde{R}_{i\bar{j}} dz^i \wedge dz^{\bar{j}}$
represent the first Chern Class of our manifold $M$. Recall also that
$\tilde{R}_{i\bar{j}} = - \frac{\partial^2}{\partial z^i \partial z^{\bar{j}}}
[\log\det(g_{k\bar{l}})]$. Since $[R - \tilde{R}]$ = 0, another application of
the Global $\partial\bar{\partial}$-Lemma demonstrates that there exists a
smooth function $F$ defined on $M$ such that: $\tilde{R}_{i\bar{j}} - 
R_{i\bar{j}} = \frac{\partial^2 F}{\partial z^i \partial z^{\bar{j}}}$. Hence:
\\
\begin{align}
\partial\bar{\partial}\log\frac{\det(\tilde{g}_{i\bar{j}})}{\det(g_{i\bar{j}})}
= \partial\bar{\partial}F.
\end{align}
\\

Since $M$ is a compact manifold, and
$\frac{\det(\tilde{g}_{i\bar{j}})}{\det(g_{i\bar{j}})}$ is a globally defined
function (this can be demonstrated by looking at the metric locally and applying
a change of coordinates), a further application of the Global
$\partial\bar{\partial}$-Lemma  shows that
$\frac{\det(\tilde{g}_{i\bar{j}})}{\det(g_{i\bar{j}})} - F$ is a constant
function. Hence there exists a constant $C > 0$ such that:
\\
\begin{align}
\det(\tilde{g}_{i\bar{j}}) = C\exp{(F)}\det(g_{i\bar{j}}).
\end{align}
\\
Recalling the fact that our K\"{a}hler forms are cohomologous we find that the
above equation is equivalent to:
\\
\begin{align}
\det\left(g_{i\bar{j}}+\frac{\partial^2 \varphi}{\partial z^i \partial
z^{\bar{j}}}\right) = C\exp{(F)}\det(g_{i\bar{j}}).
\end{align}
\\
This equation is a nonlinear, elliptic, second-order partial differential
equation in $\varphi$. It is of Monge-Amp\`{e}re type (Note: Ellipticity follows
from the fact that the metric is assumed to be Hermitian).  We now run the
argument backwards and argue that if we can find a constant $C >$ 0 such that
there exists a smooth solution $\varphi$ satisfying the integrability condition
$\int_M \varphi = 0$ and 
$\tilde{g} = \sum_{i,j} g_{i\bar{j}}+\frac{\partial^2 \varphi}{\partial z^i
\partial z^{\bar{j}}} dz^i \otimes dz^{\bar{j}}$ defines a K\"{a}hler Metric
then we will have a solution to the Calabi Conjecture. Note the integrability
condition is imposed so that $\varphi$ is chosen to be unique. Not imposing this
integrability condition gives us a solution to the Calabi Conjecture up to a
constant.    

Moreover one readily sees that the constant $C$ must satisfy the following
compatibility condition:
\\
\begin{align}
C\int_M \exp{ \{F\}} = \textrm{Vol}(M).
\end{align}
\\
We now have an equivalent formulation of the Calabi Conjecture:
\begin{theorem}
(Calabi Conjecture Reformulation) Let M be a compact, complex manifold and g a
K\"{a}hler metric on M with K\"{a}hler form $\omega$. Let F be a smooth function
on M and let C $>$ 0 be defined via the compatibility condition $C\int_M \exp{
\{F\}} = Vol(M)$. Then $\exists$ unique smooth function $\varphi$ such that:
\\
(i) $\tilde{g} = \sum_{i,j} g_{i\bar{j}}+\frac{\partial^2 \varphi}{\partial z^i
\partial z^{\bar{j}}} dz^i \otimes dz^{\bar{j}}$ defines a K\"{a}hler Metric 
\\
(ii) $\int_M \varphi = 0$
\\
(iii) $\det\left(g_{i\bar{j}}+\frac{\partial^2 \varphi}{\partial z^i \partial z^{\bar{j}}}\right) = C\exp{(F)}\det(g_{i\bar{j}})$
\end{theorem}

Note: Assuming such a $\varphi$ exists, (i) of the theorem follows from (iii)
and the lemma proved in the last section. Before closing the present section we
prove that our solution to the Complex Monge Amp\`{e}re equation is unique. 
\\
\begin{theorem}
(Uniqueness of K\"{a}hler Metric) Under the assumptions of the previous theorem,
the smooth function $\varphi$ is unique.
\end{theorem}   
\begin{proof}
Suppose we have $\varphi_{1}$ and $\varphi_{2}$ solving the Complex Monge
Amp\`{e}re Equation. Further assume that $\varphi_{1}$, $\varphi_{2}$ $\in
C^{3}(M)$. Let $\omega_{1} = \omega + i\partial \overline{\partial} \varphi_{1}$
and $\omega_{2} = \omega + i\partial \overline{\partial} \varphi_{2}$. By the
lemma in the previous section we know that both $\omega_{1}$ and $\omega_{2}$
are positive (1,1)-forms. Let $g_{1}$ and $g_{2}$ be the respective K\"{a}hler
Metrics.  
\\
We introduce the operator $d^{c} = i(\bar{\partial} - \partial)$. We also
introduce the Hodge Star operator $\star$. Let $V$ be an oriented vector space
of dimension $m$. Furthermore let $\{e_{i}\}$ be a positive oriented orthonormal
basis for $V$.  Then Hodge Star operator $\star$ is defined by $\alpha \wedge
\star \beta = g(\alpha,\beta)\textnormal{vol}$. 
\\  

We have:
\\
\begin{equation}
0 = \omega_{1}^{m} - \omega_{2}^{m} = dd^{c}\left(\varphi_{1} -
\varphi_{2}\right) \wedge \sum_{k = 0}^{m-1} \omega^{k}_{1} \wedge
\omega^{m-k-1}_{2}.  
\end{equation}
\\
Multiply both sides by $\left(\varphi_{1} - \varphi_{2}\right)$ We have:
\begin{equation}
0 = \left(\varphi_{1} - \varphi_{2}\right) dd^{c}\left(\varphi_{1} -
\varphi_{2}\right) \wedge \sum_{k = 0}^{m-1} \omega^{k}_{1} \wedge
\omega^{m-k-1}_{2},
\end{equation}
\\
which in turns implies the following:
\\
\begin{equation}
0 = d\left[\left(\varphi_{1} - \varphi_{2}\right)d^{c}\left(\varphi_{1} -
\varphi_{2}\right) \wedge\sum_{k = 0}^{m-1} \omega^{k}_{1} \wedge
\omega^{m-k-1}_{2}\right] - d\left(\varphi_{1} - \varphi_{2}\right) \wedge d^{c}
\left(\varphi_{1} - \varphi_{2}\right) \wedge \sum_{k = 0}^{m-1} \omega^{k}_{1}
\wedge \omega^{m-k-1}_{2}.
\end{equation}
\\
We integrate over $M$ and use Stokes Theorem to get:
\begin{equation}
\sum_{k = 0}^{m-1} \int_M d\left(\varphi_{1} - \varphi_{2}\right) \wedge d^{c}
\left(\varphi_{1} - \varphi_{2}\right) \wedge \omega^{k}_{1} \wedge
\omega^{m-k-1}_{2} = 0.
\end{equation}
\\

Since $\omega_{1}$ and $\omega_{2}$ define K\"{a}hler Metrics $g_{1}$ and
$g_{2}$ we can choose a local holomorphic coordinate system around each $x \in
M$ where $\{e_{1} \dots e_{m},Je_{1} \dots Je_{m}\}$ is an orthonormal basis
with $e_{\alpha} = \frac{\partial}{\partial x_{\alpha}}(x)$ and $Je_{\alpha} = 
\frac{\partial}{\partial y_{\alpha}}(x)$. Also:
\\
\begin{align}
\omega_{1} = \sum_{j = 1}^{m} e_{j} \wedge Je_{j}
\end{align}
\begin{align}
\omega_{2} = \sum_{j = 1}^{m} a_{j}e_{j} \wedge Je_{j}
\end{align}
\\
Where $a_{j}$ are strictly positive local functions.  This shows that:
\begin{equation}\omega^{k}_{1} \wedge \omega^{m-k-1}_{2} = \star \left(\sum_{j =
1}^{m} b^{k}_{j}e_{j} \wedge Je_{j} \right) \textrm{ where } b^{k}_{j} =
k!(m-k-1)! \sum_{\stackrel{j_{1}\neq j \dots j_{k} \neq j} {j_{1} < \dots <
j_{k}}} a_{j_{1}} \dots a_{j_{k}}.\end{equation}
\\

This shows that the integrand is strictly positive unless $0 =
d\left(\varphi_{1} - \varphi_{2}\right)$. Thus $\varphi_{1} - \varphi_{2}$ is a
constant. Since we imposed the integrability condition the constant is equal to
0. Hence $\varphi_{1} = \varphi_{2}$.     
\end{proof}

In the next section we look at some of the important results from Elliptic PDE
theory that helped Yau in proving the Calabi Conjecture.

\section{Elliptic PDE Theory}
The theory of Partial Differential Equations is enormously varied, yet a
consistent strategy employed in solving both quasilinear and nonlinear PDE's has
been to prove that solutions to the PDE must satisfy certain $\textit{a priori}$
estimates. When one assumes that a solution to a PDE lies in a specific function
class one is in fact imposing a regularity condition on solutions to the PDE.
Proving an $\textit{a priori}$ estimate amounts to showing that under the basic
assumption that a solution to a PDE belongs to a given function class (i.e. has
sufficient regularity) that we can in fact find a uniform bound for all
solutions of that function class. This of course says nothing about the actual
existence of a solution to a particular PDE.  
\\

To prove the existence of a solution to a quasilinear or nonlinear PDE, one
employs fixed point methods or continuity arguments that links the nonlinear PDE
under consideration to a linear PDE for which one can show the existence of
sufficiently regular solutions. In the previous section it was demonstrated
that a solution to the Calabi Conjecture would follow from proving the existence
of a smooth solution to a nonlinear elliptic PDE of Monge-Amp\`{e}re type. Much
of the theory we discuss in this section will help us understand the structure
of Yau's argument and help to prove existence and uniqueness of a smooth
solution. 
\\

Let us recall the general Second Order Linear Elliptic PDE: 
\\
\begin{align}
\sum_{i,j}a_{ij}\partial_{ij}u + \sum_{i}b_{i}\partial_{i}u + cu = f.
\end{align}
\\
We rewrite this as $Lu = f$ where $L$ is a Second Order Linear Elliptic
Operator. Ellipticity follows from the fact that the coefficient matrix
$[a_{ij}]$ is positive definite in the domain of the respective arguments. For
what follows we assume that we are working in an open domain $\Omega \subseteq
\mathbb{R}^{n}$. We also assume that $c \leq 0$. Furthermore we assume that our
operator is $\textit{Uniformly Elliptic}$: $\forall x \in \Omega, \xi \in
\mathbb{R}^{n}, \lambda > 0$. 
\\
\begin{align}
\sum_{i,j}a_{ij}(x) \xi_{i}\xi_{j} \geqq \lambda \left|\xi\right|^{2}.
\end{align}
\\
We start by defining some function classes that will be of importance in our
analysis:
\begin{definition}
A function $f$ is uniformly H\"{o}lder Continuous with exponent $0 < \alpha
< 1$ if $$\left[f\right]_{\alpha;\Omega} = \sup_{\stackrel{x,y\in \Omega}{x \ne
y}} \frac{\left|f(x) - f(y)\right|}{\left|x - y\right|^{\alpha}} < \infty.$$ 
\end{definition}
We define the $H\ddot{o}lder$ $Spaces$ $C^{k,\alpha}(\bar{\Omega})$
($C^{k,\alpha}({\Omega}$)) to be the function space consisting of functions
whose $k^{th}$ order partial derivatives are uniformly  H\"{o}lder
Continuous. Note: $C^{k,0}(\bar{\Omega}) = C^{k}(\bar{\Omega}) $
($C^{k,0}({\Omega}) = C^{k}(\Omega))$ where $C^{k}(\Omega)$
($C^{k}(\bar{\Omega})$) is the space of continuous functions whose $k$th order
partial derivatives are continuous (up to the boundary). We define the following
norms: 
\\
\begin{align}
\|u\|_{C^{k}(\bar{\Omega})}             &= \sum_{j=0}^{k}
\left[D^{j}u\right]_{0; \bar{\Omega}}, \\
\|u\|_{C^{k,\alpha}(\bar{\Omega})} &= \|u\|_{C^{k}(\bar{\Omega})}  +
[D^{k}f]_{\alpha;\bar{\Omega}}.
 \end{align}
 \\

To understand the estimates below it is important to make a comment in regards
to how one goes about obtaining them. One begins by first finding  $\textit{a
priori}$  estimates on suitable subdomains $\Omega^{'} \subseteq \subseteq
\Omega$. Since one can go about establishing a Maximum Principle for Linear
Elliptic PDEs one will find that solutions to such PDEs will have their
maximum on the boundary $\partial\Omega$. Hence \textit{if the domain $\Omega$
has sufficiently smooth boundary and the solution to our PDE is sufficiently
regular, one can extend these interior estimates to the boundary and obtain
global estimates.}
\begin{definition}
A bounded domain $\Omega \subset \mathbb{R}^{n}$ and its boundary are of class
$C^{k,\alpha}$ if at each point $x_{0} \in \partial\Omega$ there is a ball $B =
B(x_{0})$ and a one-to-one mapping $\psi$ of $B$ onto $D\subset \mathbb{R}^{n}$
such that: 
\\
(i) $\psi(B \cap \Omega)$, 
\\
(ii) $\psi(B \cap \partial \Omega) \subset \partial \mathbb{R}^{n}_{+}$,
\\
(iii) $\psi \in C^{k,\alpha}(B), \psi^{-1} \in C^{k,\alpha}(D)$.
\end{definition}

One must refer to Potential Theory to establish existence and uniqueness of
solutions for the Laplace and Poisson Equation. But we mention that one can
prove the unique existence of a $C^{2,\alpha}(\bar{\Omega}$) solution for the
Poisson Equation $\Delta u = f$ when $f \in C^{0,\alpha}(\bar{\Omega})$ . In the
course of this problem one in fact obtains an  $\textit{a priori}$ estimate for
this unique solution: 
\\
\begin{align}
\|u\|_{C^{2,\alpha}(\bar{\Omega})}  \leqq  C\left(\sup_{\Omega} \left|u\right| +
\|f\|_{C^{0,\alpha}(\bar{\Omega})}\right) \textrm{ where } C = C(n,\alpha).
\end{align}

The Poisson Equation is a Linear Elliptic PDE with constant coefficients, yet
one can show that the estimate above
can be extended to the case of the general Linear Second Order Elliptic PDE with
variable coefficients. In fact one can show even more: that if the coefficients
of the general Second Order Linear PDE are H\"{o}lder Continuous then $\exists$
an $\textit{a priori}$ estimate for $C^{2,\alpha}(\bar{\Omega})$ solutions for
this class of PDE's. \textit{The idea is to treat the equation locally as a
perturbation of constant coefficient equations.} Similar in spirit to the
estimate obtained above for the Poisson Equation, one first considers interior
estimates and then assuming sufficient regularity of our solutions and
sufficient smoothness up to the boundary, one then extends these estimates to
the boundary to obtain global estimates. 
\begin{theorem}
(Schauder Estimate) Let $\Omega$ be a $C^{2,\alpha}$ domain in $\mathbb{R}^{n}$
and let $u \in C^{2,\alpha}(\bar{\Omega})$ be a solution of $Lu = f$ in $\Omega$
 where $f \in C^{0,\alpha}(\bar{\Omega})$ and coefficients of L satisfy for
positive constants $\lambda$, $\Lambda$:
\\

$a_{ij}(x) \xi_{i}\xi_{j} \geqq \lambda \left|\xi\right|^{2}$, $\forall x \in
\Omega, \xi \in \mathbb{R}^{n}, \lambda > 0 \; \textnormal{  (Ellipticity
Condition) }$ 
\\

$[a_{ij}]_{\alpha; \bar{\Omega}}$, $[b_{i}]_{\alpha; \bar{\Omega}}$,
$[c]_{\alpha; \bar{\Omega}}$ $\leq$ $\Lambda \; \textnormal{ (Coefficients are
H\"{o}lder Continuous)}$
\\

Assume further that $\varphi \in C^{2,\alpha}(\bar{\Omega})$ and $u = \varphi$
on $\partial \Omega$. Then,
\\
 $$ \|u\|_{C^{2,\alpha}(\bar{\Omega})}  \leqq  C\left(\|u\|_{0,\bar{\Omega}} +
\|\varphi\|_{2,\alpha;\bar{\Omega}} + \|f\|_{0,\alpha; \bar{\Omega}}\right)$$
\begin{center}
$C = C(n,\alpha, \lambda, \Lambda, \Omega)$
\end{center}
\end{theorem}

We now have an $\textit{a priori}$ estimate for solutions in the class
$C^{2,\alpha}$. What remains to be considered is to show that there in fact does
exist a unique $C^{2,\alpha}$ solution to our second order Linear Elliptic PDE.
To prove existence of a solution we will use functional analytic methods. The
following is known as the \textit{Linear Continuity Method}. 
\begin{theorem}
(Linear Continuity Method) Let $(B_{1}, \|.\|_{B_{1}})$ and $(B_{2},
\|.\|_{B_{2}})$ be Banach Spaces. Let $L_{0}, L_{1}: B_{1} \to B_{2}$ linear and
bounded.  Set $L_{t} = (1-t)L_{0} + tL_{1}$.
\\
Furthermore $\|x\|_{B_{1}} \leqslant C\|L_{t}x\|_{B_{2}}$ $\forall x \in X$,
$\forall t \in [0,1]$.
Then the following are equivalent:
\begin{itemize}
 \item  $L_{0}: B_{1} \to B_{2}$ isomorphism.
\item $L_{1}: B_{1} \to B_{2}$ isomorphism.
\item $ \exists t_{0} \in [0,1]$ such that
$L_{t_{0}}: B_{1} \to B_{2}$ isomorphism. 
\end{itemize}
\end{theorem}   
\begin{proof}
Assume $L_{s}$ is onto for some $s \in [0,1]$. By the bound in the assumption we
have that $L_{s}$ is also injective hence $L_{s}$ is an isomorphism. 
\\
$\Rightarrow$ $L^{-1}_{s} : B_{2} \to B_{1}$ exists.
\\
$\Rightarrow$ $\forall y \in B_{2}$ $\exists x \in B_{1}$ s.t $y = L_{s}x$.
\\
$\Rightarrow$ $\|L^{-1}_{s}y\|_{B_{1}} = \|x\|_{B_{1}}  \leqslant
C\|L_{s}x\|_{B_{2}} = C\|y\|_{B_{2}}$
\\

Let $t\in [0,1]$. Now we know that:
\begin{align}
  L_{t}x = y &\iff L_{s}x = (L_{s} - L_{t})x + y = y + (t-s) L_{0}x -
(t-s)L_{1}x.
\\ &\iff x = L^{-1}_{s}y + (t-s)L^{-1}_{s}(L_{0} - L_{1})x.
\end{align}

Define map $T: X \to X$ such that  $T(x) = L^{-1}_{s}y + (t-s)L^{-1}_{s}(L_{0} -
L_{1})x$.
\\
We have a string of inequalities: 
\begin{align}\|Tx - Tx^{'}\|_{B_{1}} &\leqslant 
\left|t-s\right| \|L^{-1}_{s}(L_{0} - L_{1}) (x -x')\|_{B_{1}} \\ 
&\leqslant \left|t-s\right| C (\|L_{0}\| + \|L_{1}\|)\|x -x'\|_{B_{1}}
\end{align}

If we choose $t$ close enough to $s$ we see that from the last inequality that
the map $T$ is contractive. In particular if $ \left|t-s\right| \leqslant
\frac{1}{2 (\|L_{0}\| + \|L_{1}\|)} =: \delta$ we see that:
\\

$\|Tx - Tx^{'}\|_{B_{1}} \leqslant \frac{1}{2}\|x -x'\|_{B_{1}}$.   
\\

Hence by the Banach Fixed Point Theorem:
\\

If $\left|t-s\right| \leqslant \delta$ then $\exists !$ fixed point for $T$ 
$\Rightarrow$ $L_{t}$ is onto $\forall t$, $\left|t-s\right| \leqslant \delta$.
\\
Since our map is surjective in a uniform neighborhood we can iterate our
argument to show surjectivity for $L_{t}$ $\forall t \in [0,1]$. 
\end{proof}

We now establish a Maximum Principle for solutions of 2nd Order Linear Elliptic
PDE. This will allows us to successfully apply the Linear Continuity Method and
prove existence of a unique $C^{2,\alpha}$ solution.
\begin{lemma}
(Maximum Principle) Assume $Lu = f$ in bounded domain $\Omega \subseteq
\mathbb{R}^{n}$. Suppose $u \in C^{0}(\bar{\Omega}) \cap C^{2}(\Omega)$, $c \leq
0$, $a_{ij} \geq \lambda\mathbb{I}$, $b_{i} \in L^{\infty}$. Then,
$$\sup_{\Omega} u \leq \sup_{\partial \Omega} u^{+} + \frac{C_{0}\sup_{\Omega}
f^{-}}{\lambda}$$
\begin{center}
Where $u^{+}$ is positive part of function u, $f^{-}$ is negative part of
function f, and $C_{0} = C_{0}(\Omega, \lambda, \|b\|_{L^{\infty}})$.
\end{center}
\end{lemma}
\begin{proof}
Let $L_{0}$ = $a_{ij} \partial_{ij} + b_{i} \partial_{i}$
\\
We translate and rotate $\Omega$ such that in the $x_{1}$ direction it is
bounded between $[0,d]$.
\\
We know that $L_{0}e^{\alpha x_{1}} = (a_{11}\alpha^{2} + b_{1}\alpha)e^{\alpha
x_{1}}$
$\geqslant (\lambda \alpha^{2} - \|b\|_{L^{\infty}} \alpha) e^{\alpha x_{1}}
\geqslant \lambda$ if $\alpha > > 1$. 
\\

Set $v := \sup_{\partial \Omega} u^{+} + (e^{\alpha d} - e^{\alpha
x_{1}}) \frac{\sup_{\Omega} f^{-} + \varepsilon}{\lambda} \geqslant 0$
\\
Now $L_{0}v = -\lambda \frac{\sup_{\Omega} f^{-} + \varepsilon}{\lambda}
\leqslant - \sup_{\Omega} f^{-} - \varepsilon.$
\\

Since $v \geqslant 0$ and $c \leqslant 0$, $Lv = L_{0}v + cv \leqslant L_{0}v$. 
\\
$\Rightarrow$  $L(v-u) \leqslant -\lambda \left(\frac{\sup_{\Omega} f^{-} +
\varepsilon}{\lambda}\right) \leqslant 0$ and $v-u \geqslant 0$ on $\partial
\Omega$.
\\

We want to show that our functions $u$ and $v$ never cross. 
\\
Let $x_{0}$ be a maximum point of $v-u$. Assume by contradiction that
$(v-u)(x_{0}) < 0$. 
\\

$\Rightarrow$ $L(v-u)(x_{0}) = a_{ij} \partial_{ij} (v-u)(x_{0}) + b_{i}
\partial_{i}(v-u)(x_{0}) + c(v-u)(x_{0}) \geqslant 0$
\\
But $L(v-u)(x_{0}) \leqslant 0$. Hence we have derived a contradiction.
\\
$\Rightarrow$ In $\Omega$: $u \leqslant v$. Let $\varepsilon \to 0$ to obtain
bound.  
\end{proof}

With the help of the Maximum Principle and the Linear Continuity Method we prove
the existence of a unique $C^{2,\alpha}$ solution for 2nd Order Linear Elliptic
PDE.
\begin{theorem}
(Schauder) Assume $\Omega \subseteq \mathbb{R}^{n}$ and its boundary $\partial
\Omega$ is of class $C^{2,\alpha}$. Suppose also that $a_{ij}$, $b_{i}$, $c$
$\in C^{0,\alpha}(\Omega)$ ; $f\in C^{0,\alpha}(\Omega)$ ;  $a_{ij} \geq
\lambda\mathbb{I}$, $b_{i} \in L^{\infty}$,and $c \leq 0$. Then, $\forall
\varphi \in C^{2,\alpha}(\Omega)$ $\exists!$ $u \in C^{2,\alpha}(\Omega)$
solving the Boundary Value Problem:
\begin{center}
$Lu = f$ in $\Omega$
\end{center}
\begin{center}
$u = \varphi$ on $\partial\Omega$
\end{center}
\end{theorem}
\begin{proof} Without loss of generality we assume that $\varphi = 0$. Otherwise
we can replace u with $v = u-\varphi$. Define $L_{t} = (1-t) \Delta + tL$. 
\\
Note that  $L_{t}: X \to Y$, where $X = \{u \in C^{2,\alpha}(\Omega) | u = 0 \;
\text{on} \; \partial\Omega \}$ and $Y = \{C^{0,\alpha}(\Omega)\}$   
\\

By the maximum principle $\forall w \in X$ $\|w\|_{L^{\infty}(\Omega)} \leqslant
C\|L_{t}w\|_{L^{\infty}(\Omega)}$.
\\

Our $C^{2,\alpha}$ $\textit{a priori}$ estimate implies that:
\\$\|w\|_{C^{2,\alpha}(\bar{\Omega})}  \leqq  C\left(\sup_{\Omega}
\left|w\right| + \|f\|_{C^{0,\alpha}(\bar{\Omega})}\right) \leqslant
C\|L_{t}w\|_{C^{0,\alpha}(\Omega)} $.  
\\
$\Rightarrow$ $\|w\|_{X}  \leqq C\|L_{t}w\|_{Y}$ $\forall w \in X$, $\forall t
\in [0,1]$.
\\
We know that $L_{0}$ is simply the Laplacian operator and we stated that there
exists a unique $C^{2,\alpha}$ solution for this operator. Hence by the
continuity method $L_{1}$ is an isomorphism.\end{proof}

In order to successfully study the Complex Monge Amp\`{e}re Equation we have to
also consider Nonlinear Elliptic PDE theory. Let us recall the General
Second-Order Nonlinear Elliptic PDE on domain $\Omega \subseteq \mathbb{R}^{n}$.
Our nonlinear operator will be a real function defined on $\Gamma = \Omega
\times \mathbb{R} \times \mathbb{R}^{n} \times \mathbb{R}^{n\times n}$: 
\\
\begin{align}
F[u] = F(x,u,Du, D^{2}u) = 0.
\end{align}
\begin{center}
A typical point $\gamma \in \Gamma$ will be indexed by $\gamma = (x,z,p,r)$.
\end{center}
\begin{definition}
F is elliptic in $\mathfrak{A} \subseteq \Gamma$ if this matrix
$[F_{ij}(\gamma)]$ given by $F_{ij}(\gamma) = \frac{\partial F}{\partial
r_{ij}}(\gamma) > 0$ $\forall \gamma \in \mathfrak{A} $. Furthermore let
$\lambda(\gamma)$ and $\Lambda(\gamma)$ be the minimum and maximum eigenvalues
of matrix $[F_{ij}(\gamma)]$. F is uniformly elliptic if
$\frac{\Lambda}{\lambda} < \infty$.
\end{definition}
To show the existence of a solution to a Nonlinear Elliptic PDE one can use a
nonlinear version of the Continuity Method. We will now develop this idea
further. We assume the reader is familiar with Fr\'{e}chet Differentiability for
operators on a Banach Space. 
\begin{theorem}
(Implicit Function Theorem for Banach Spaces) Let $B_{1}, X, B_{2}$ be Banach
Spaces and suppose that $G: B_{1} \times X \to B_{2}$ is a $C^{1}$ mapping
defined at least in a neighborhood of a point $(u_{0}, \sigma_{0})$. Denote by
$y_{0}$ the image of $G(u_{0}, \sigma_{0})$. Suppose $D_{x}G(u_{0}, \sigma_{0})$
is an isomorphism. Then $\exists$ open sets $W \subseteq B_{1}$, $U \subseteq
X$, $V\subseteq B_{2}$ with $u_{0} \in W$, $\sigma_{0} \in U$ and $y_{0} \in V$
and a unique $C^{1}$ mapping $g: W \times V \to U$ such that $$G(u,g(u,y)) =
y$$ 
\begin{center}
$\forall (u,y) \in W \times V.$
\end{center}
\end{theorem}
In order to apply the Implicit Function Theorem to nonlinear PDE Theory we
assume that F is a mapping from an open subset $\mathfrak{A} \subseteq B_{1}$
into $B_{2}$. Let $\psi$ be a fixed element in $\mathfrak{A}$ and define for $u
\in \mathfrak{A}$, $t \in \mathbb{R}$ the mapping $G: \mathfrak{A} \times
\mathbb{R} \to B_{2}$ where $$G[u,t] = F[u] - tF[\psi].$$
Define $S  \subseteq [0,1]$ such that:

$$S = \{t \in [0,1] \; |\; G[u,t] = 0\; \textnormal{for some}\;u \in
\mathfrak{A} \}.$$

\textit{Note:} $S \neq \varnothing$ since $t = 1 \in S$. If we further assume
that the map $F$ is $C^{1}$ it follows from the Implicit Function Theorem that
the set $S$ is open. If we can show that the set $S$ is also closed then by
connectivity of the set $[0,1]$, $S = [0,1]$. Hence in particular there exists a
$u \in \mathfrak{A}$ such that $F[u] = 0$. This is the solution to our Nonlinear
PDE. As we will see when we apply these ideas to the Complex Monge Amp\`{e}re
Equation, closure of the set $S$ will follow from an $\textit{a priori}$
estimate in some function space and the application of the Arzela-Ascoli
Theorem. Hence just as in the Linear case we need to establish $\textit{a
priori}$ estimates for solutions with sufficient regularity. Once these
estimates have been shown a simple application of the Continuity Method gives us
a solution to the Nonlinear PDE. As a precursor for what is to come, we mention
that the establishment of $\textit{a priori}$ estimates for $\varphi$ is Yau's
primary task. He then successfully applies a variant of the Nonlinear Continuity
Method to prove existence of a solution to the nonlinear elliptic PDE under
consideration. Hence most of the labor Yau undertakes is in establishing
$\textit{a priori}$ estimates for solutions of the Complex Monge Amp\`{e}re
Equation.

\section{Proof of Calabi Conjecture}
In this section we present a proof of the Calabi Conjecture. Before providing
the details of the proof let us take a moment to mention how our results of the
previous section generalize to general compact manifolds.  The interior
estimates we obtained were for general open domains in $\mathbb{R}^{n}$. To
transfer these estimates onto a compact manifold we simply use our compactness
assumption to find a finite open cover for our manifold. Since each set in this
cover is diffeomorphic to an open set in $\mathbb{R}^{n}$, one can simply use
coordinate transformations to transfer the estimate onto our manifold.
\\

A tool that will be useful in our computations is the fact that around each
point of our K\"{a}hler Manifold one can find a coordinate system which can
simultaneously diagonalize the K\"{a}hler Metric $g_{ij}$ and the Hessian of
$\varphi$. The utility of this representation can hardly be overestimated.  
More specifically:
\begin{theorem} (Existence of Holomorphic Normal Coordinates)
Let $(M,\omega)$ be a K\"{a}hler Manifold. In local coordinates $\omega = 
\sum_{i,j} g_{i\bar{j}} dz^i \wedge dz^{\bar{j}}$. At each $z_{0} \in M$
holomorphic normal coordinates can be introduced i.e.
\\ 
(i) $g_{i\bar{j}}(z_{0}) = \delta_{ij}$ 
\\
(ii) $\partial_{k}g_{i\bar{j}}(z_{0}) = \partial_{\bar{k}}g_{i\bar{j}}(z_{0}) =
0$
\\
(iii) $g'_{i\bar{j}}(z_{0}) = \left(1 +
\varphi_{i\bar{i}}\right)\delta_{i\bar{j}}$ where $g'_{i\bar{j}} = 
g_{i\bar{j}}+\frac{\partial^2 \varphi}{\partial z^i \partial z^{\bar{j}}}$ 
\end{theorem}  

We start by stating the main estimates which are used to provide a solution to
the Calabi Conjecture.
\\
\begin{theorem} (Yau's Second Order Estimates) Let M be a compact K\"{a}hler
Manifold with metric tensor 2 $\sum_{i,j} g_{i\bar{j}} dz^i \otimes
dz^{\bar{j}}$. Let $\varphi$ be a real valued function in $C^{4}(M)$ such that
$\int_M \varphi = 0$ and $\sum_{i,j} g_{i\bar{j}}+\frac{\partial^2
\varphi}{\partial z^i \partial z^{\bar{j}}} dz^i \otimes dz^{\bar{j}}$ defines
another metric tensor on M. Suppose $\det\left(g_{i\bar{j}}+\frac{\partial^2
\varphi}{\partial z^i \partial z^{\bar{j}}}\right) =
\exp{(F)}\det(g_{i\bar{j}})$. Then there exist positive constants $C_{1}$,
$C_{2}$, $C_{3}$, and $C_{4}$ depending on $\inf_{M}F$, $\sup_{M}F$, $\inf
\Delta F$, and M such that $\sup_{M} \left|\varphi \right| \leqslant C_{1}$,
$\sup_{M} \left|\triangledown \varphi \right| \leqslant C_{2}$, $0 < C_{3}
\leqslant 1 + \varphi_{i\bar{i}} \leqslant C_{4}$ for all i.         
\end{theorem}

\begin{theorem} (Yau's Third Order Estimate) Let M be a compact K\"{a}hler
Manifold with metric tensor $\sum_{i,j} g_{i\bar{j}} dz^i \otimes dz^{\bar{j}}$.
Let $\varphi$ be a real valued function in $C^{5}(M)$ such that $\int_M \varphi
= 0$ and $\sum_{i,j} g_{i\bar{j}}+\frac{\partial^2 \varphi}{\partial z^i
\partial z^{\bar{j}}} dz^i \otimes dz^{\bar{j}}$ defines another metric tensor
on M. Suppose $\det\left(g_{i\bar{j}}+\frac{\partial^2 \varphi}{\partial z^i
\partial z^{\bar{j}}}\right) = \exp{(F)}\det(g_{i\bar{j}})$. Then there is an
estimate of the derivatives $\varphi_{i\bar{j}k}$ in terms of $\sum_{i,j}
g_{i\bar{j}} dz^i \otimes dz^{\bar{j}}$, $\sup \left|F\right|$, $\sup
\left|\triangledown F\right|$, $\sup_{M} \sup_{i} \left|F_{i\bar{i}}\right|$ and
$\sup_{M} \sup_{i,j,k} \left|F_{i\bar{j}k}\right|$.
\end{theorem}

We consider
\\
\begin{align}
\det\left(g_{i\bar{j}}+\frac{\partial^2 \varphi}{\partial z^i \partial
z^{\bar{j}}}\right) = C\exp{(F)}\det(g_{i\bar{j}})
\end{align} 
\\
Recall our assumptions: $\varphi \in C^5(M)$, F $\in C^{k}(M)$ for $k \geq 3$
and $\int_M \exp{ \{F\}} = \textrm{Vol}(M)$. We set our constant $C = 1$. Our
goal is to show the existence of a unique $C^{\infty}$ function $\varphi$
solving the Monge Amp\`{e}re equation and  satisfying the compatibility
condition $\int_M \varphi = 0$ . Recall that we have already established that if
such a smooth function $\varphi$ does exist then $\tilde{g} = \sum_{i,j}
g_{i\bar{j}}+\frac{\partial^2 \varphi}{\partial z^i \partial z^{\bar{j}}} dz^i
\otimes dz^{\bar{j}}$ indeed defines a K\"{a}hler metric on our manifold.
Moreover we showed that the metric is unique. Hence establishing the existence
of a smooth solution to the Complex Monge Amp\`{e}re equation will lead us to a
solution of the Calabi Conjecture. We now demonstrate how Yau's estimates and an
application of a variant of the continuity method will lead us to a solution of
the problem. 
\\

In the first step we show that under the stated assumptions we can find a
solution $\varphi \in C^{k+1, \alpha}(M)$ for any $0 \leq \alpha < 1$ to the
Complex Monge Amp\`{e}re Equation. Consider the set:
$$\mathbf{S} = \{t \in [0,1] \; |\; \textnormal{the equation} \;
\det\left(g_{i\bar{j}}+\frac{\partial^2 \varphi}{\partial z^i \partial
z^{\bar{j}}}\right) \det(g_{i\bar{j}})^{-1} =$$ $$\textrm{Vol}(M)\left[\int_M
\exp{ \{tF\}}\right]^{-1}\exp{(tF)} \; \textnormal{has a solution in} \;
C^{k+1,\alpha}(M) \}$$
\begin{center}
Note: $0 \in \mathbf{S} \Rightarrow \mathbf{S} \neq \varnothing$.
\end{center}
If we can show that $\mathbf{S}$ is both open and closed this will imply that $S
= [0,1]$. Hence $\Rightarrow$ $1 \in \mathbf{S} \Rightarrow$ our equation has a
solution in $C^{k+1, \alpha}(M)$.  This is an application of the Nonlinear
Continuity Method. Define the following sets: 
\\
\begin{align}
\mathbf{\Theta} = \{\varphi \in C^{k+1,\alpha}(M) \;|\; 1 + \varphi_{i\bar{i}} >
0 \; \forall i \; \textnormal{and} \; \int_M \varphi = 0\}
\end{align}  
\begin{align}
\mathbf{B} = \{f \in C^{k-1,\alpha}(M) \;|\; \int_M f = \textnormal{Vol(M)}\}
\end{align}
\\
We note that $\mathbf{\Theta} \subseteq C^{k+1, \alpha}$ open and $\mathbf{B}
\subseteq C^{k-1,\alpha}$ is a hyperplane. Furthermore $C^{k+1, \alpha}$ and
$C^{k-1, \alpha}$ are Banach Spaces. 
\\

We define the Monge Amp\`{e}re map $G: \mathbf{\Theta} \to \mathbf{B}$:
\\
\begin{align}
G(\varphi) =  \det\left(g_{i\bar{j}}+\frac{\partial^2 \varphi}{\partial z^i
\partial z^{\bar{j}}}\right)(\det(g_{i\bar{j}})^{-1})
\end{align}
\\
This is a nonlinear Map between Banach Spaces. We compute its Fr\'{e}chet
Derivative. 
We let $A$ denote the matrix $\left(g_{i\bar{j}}+\frac{\partial^2
\varphi}{\partial z^i \partial z^{\bar{j}}}\right)$ and $X$ the matrix
$\left(\frac{\partial^2 \psi}{\partial z^i \partial z^{\bar{j}}}\right)$ where
$\psi \in \mathbf{\Theta}$.  We recall a fact from linear algebra that to first
order the derivative of the determinant of an invertible matrix is given by the
trace:
\\
\begin{align}
\det(A + \varepsilon X) - \det(A) = \det(A)\textnormal{Tr}(A^{-1}X)\varepsilon
+ o(\varepsilon^{2})
\end{align}
\\

Furthermore on a K\"{a}hler Manifold the Laplace-Beltrami Operator is locally
represented by:
\begin{align}
\Delta = -\sum_{i,j} g^{i\bar{j}}\frac{\partial^2}{\partial z^i \partial
z^{\bar{j}}}
\end{align}
Where $(g^{i\bar{j}})$ represents the inverse matrix of the metric coefficient
matrix $g_{i\bar{j}}$.
\\

Hence the differential of $G$ at the point $\varphi_{0}$ is given by:
\\
\begin{align}
G'(\varphi_{0}) = \det\left(g_{i\bar{j}}+\frac{\partial^2 \varphi_{0}}{\partial
z^i \partial z^{\bar{j}}}\right)(\det(g_{i\bar{j}}))^{-1} \Delta_{\varphi_{0}}
\end{align}
Where $G': T_{\psi}\mathbf{\Theta} \to T_{G(\psi)}\mathbf{B}$ is a map between
the respective tangent spaces and  $\Delta_{\varphi_{0}}$ is the
Laplace-Beltrami Operator with respect to the metric
$\left(g_{i\bar{j}}+\frac{\partial^2 \varphi_{0}}{\partial z^i \partial
z^{\bar{j}}}\right)$. Furthermore: 
\\
\begin{align}
T_{\gamma}\mathbf{B} \subseteq \{f \in C^{k-1,\alpha}(M) \;|\; \int_M f = 0\}
\end{align}
\\

We now state a lemma about the Laplace-Beltrami Operator on compact Riemannian
Manifolds.
\begin{lemma} Let $\Delta$ be the Laplace-Beltrami operator on a compact
Riemannain manifold $(M,g)$. Assume $f: M \to \mathbb{R}$ is a smooth function.
Then there exists unique solution (in the weak sense) $u \in \mathcal{H}$ to the
Poisson equation $\Delta u = f$ where $\mathcal{H} = \{u\in H^{1,2}(M) \;|\; 
\int_M u = 0\; \textnormal{and} \;  \int_M fu = 1 \} \iff \int_M f = 0$.
\end{lemma}

This lemma implies that the Laplace Beltrami Operator is a bijection on the
space of mean zero functions. Hence the condition we need to ensure that:
\\
\begin{align}
\det\left(g_{i\bar{j}}+\frac{\partial^2 \varphi_{0}}{\partial z^i \partial
z^{\bar{j}}}\right)(\det(g_{i\bar{j}}))^{-1} \Delta_{\varphi_{0}}\varphi = f
\end{align}
\\
has a solution is that $\int_M f = 0$. This equation is a Linear Second Order
Elliptic PDE.  Hence by Schauder Theory we know that $u \in C^{k+1, \alpha}$. 
Furthermore by requiring that $\int_M \varphi = 0$ we know that our solution is
unique. Hence the differential of $G, G'$ at $\varphi_{0}$ is an isomorphism. By
the Implicit Function Theorem for Banach Spaces $G$ maps and open neighborhood
of
$\varphi_{0}$ to an open neighborhood of $G(\varphi_{0})$. This shows that
$\mathbf{S}$ is an open set. 
\\

We now show that $\mathbf{S}$ is also a closed set. Let $\{t_{n}\}$ be an
arbitrary sequence in $\mathbf{S}$. This gives rise to a sequence
$\{\varphi_{n}\} \in C^{k+1, \alpha}$ such that:
\\
\begin{align}
\det\left(g_{i\bar{j}}+\frac{\partial^2 \varphi_{n}}{\partial z^i \partial
z^{\bar{j}}}\right) \det(g_{i\bar{j}})^{-1} = \textrm{Vol}(M)\left[\int_M \exp{
\{t_{n}F\}}\right]^{-1} \exp{(t_{n}F)}
\end{align}
\\

Differentiating this equation with respect to $z$ we have:
\\
\begin{align}
\det\left(g_{i\bar{j}}+\frac{\partial^2 \varphi_{n}}{\partial z^i \partial
z^{\bar{j}}}\right) \sum_{i,j} g_{n}^{i\bar{j}}\frac{\partial^2}{\partial z^i
\partial z^{\bar{j}}}\left(\frac{\partial \varphi_{n}}{\partial z^{p}}\right) 
= \textrm{Vol}(M)\left[\int_M \exp{ \{t_{n}F\}}\right]^{-1}
\frac{\partial}{\partial z^{p}}\left[\exp{(t_{n}F)} \det(g_{i\bar{j}})\right]
\end{align}
\\
Where $\left(g_{n}^{i\bar{j}}\right)$ is inverse matrix of
$\left(g_{i\bar{j}}+\frac{\partial^2 \varphi_{n}}{\partial z^i \partial
z^{\bar{j}}}\right)$ for all $n$.
\\

We notice that the left hand side of this equation is a Linear Second Order
Elliptic PDE with variable coefficients. In fact the coefficient matrix
$[a^{n}_{ij}]$ = $[g_{n}^{i\bar{j}}]$ = $[g_{i\bar{j}}+\frac{\partial^2
\varphi_{n}}{\partial z^i \partial z^{\bar{j}}}]$. By using Holomorphic Normal
Coordinates and applying Yau's 2nd Order Estimates we have $0 < C_{3} \leqslant
1 + \varphi_{n_{i\bar{i}}} \leqslant C_{4}$ for all $i$. Hence the eigenvalues
of the inverse matrix are bounded and the left hand side of $(81)$ is uniformly
elliptic. Yau's Third Order Estimates imply that $\varphi_{n} \in C^{2,\alpha}$.
This implies that the coefficients of the operator on the left hand side and the
functions on the right hand side are in fact H\"{o}lder continuous for every
exponent $0 \leqslant \alpha \leqslant 1$. Hence by the Schauder Estimate we
know
that we have $\textit{a priori}$ $C^{2,\alpha}$ estimate for $\frac{\partial
\varphi_{n}}{\partial z^{p}}$ $\forall n$ . Arguing in a similar fashion one
shows that we have $\textit{a priori}$ $C^{2,\alpha}$ estimate for
$\frac{\partial \varphi_{n}}{\partial z^{\bar{p}}}$ $\forall n$.       
\\

Furthermore this implies that $\varphi_{n} \in C^{3,\alpha}$. This gives us
better differentiability properties for the coefficients of our Linear Second
Order Elliptic operator. Appealing to Schauder Estimates we find $\textit{a
priori}$ $C^{3,\alpha}$ estimate for $\frac{\partial \varphi_{n}}{\partial
z^{p}}$ $\forall n$. We bootstrap this argument and iterate to find $C^{k+1,
\alpha}$ estimates for $\varphi_{n}$ $\forall n$ where the constant is
independent of $n$. Hence our sequence $\{\varphi_{n}\} $ is uniformly bounded.
One also readily sees that our functions are equicontinuous. Hence by the
Arzela-Ascoli Theorem our sequence $\{\varphi_{n}\}$ has a convergent
subsequence in $\mathbf{S}$. This implies that our set $\mathbf{S}$ is closed. 
We have now shown the existence of a $\varphi \in C^{k+1, \alpha}(M)$ for any $0
\leq \alpha < 1$ to the Complex Monge Amp\`{e}re Equation.  
\\

We are now in a position to prove the existence of a smooth solution $\varphi$ 
to the Complex Monge Amp\`{e}re Equation. Recall that in our formulation of the
Calabi Conjecture we had a compact complex manifold $M$ and a K\"{a}hler metric
$g$ on $M$ with K\"{a}hler form $\omega$. Furthermore we assumed $F$ to be a
smooth function on $M$. This implies $F \in C^{k}(M)$ $\forall k$. Hence by our
argument above we have that $\varphi \in C^{\infty}$. We have now followed Yau's
footsteps and provided an affirmative answer to the Conjecture of Calabi.

\section{Calabi Yau manifolds and Holonomy}
Let us begin by first defining some necessary differential geometric concepts we
will need, and proceed to mathematically define a Calabi-Yau Manifold and give a
proof of their existence. Since a K\"{a}hler
Manifold is really a Riemannian Manifold with a compatible Complex Structure, we
will
focus our discussion of differential geometric concepts in the context of
Riemannian Geometry.  We assume $(M,g)$ is a Riemannian Manifold with Riemannian
Metric $g$. The notion of a connection is introduced on a Riemannian Manifold so
that one has a suitable notion of differentiating vector fields. In fact it can
be seen as a kind of covariant derivative (a natural generalization of a
directional derivative from vector calculus) on the tangent bundle $TM$.     
\begin{theorem} (Existence of Unique Torsion-Free Connection)
Let $(M,g)$ be a Riemannian Manifold with Riemannian Metric g. Then there exists
a unique, torsion-free Connection $\nabla$ on $TM$ with $\nabla g = 0$ called
the Levi-Civita Connection. 
\end{theorem}

We next consider the holonomy of a connection. Intuitively holonomy is a local
representation of the curvature of our space. To understand the global geometry
of an object one can send vectors around closed loops in a space (formal notion
of parallel transport) and quantitatively measure how the initial vector and
final vector differ. This failure to preserve geometric data around closed loops
is what we mean by the holonomy of the connection. More specifically we choose
points $x,y \in M$. Let 
\\
$$ \gamma : [0,1] \to M$$
\\
be a smooth curve with $\gamma(0) = x$ and $\gamma(1) = y$. The connection on
our manifold allows us to transport vectors along this curve so that they remain
parallel with respect to the connection. We define $P_{\gamma}$ to be the
parallel
transport map. It is a linear and invertible map. Parallel transport is a
way to locally move the geometry of our manifold along a curve. Fixing a
point $x \in M$ one defines the holonomy of our connection, 
\begin{equation}
\textrm{Hol}_{x}(\nabla) = \{P_{\gamma} \; | \; \gamma \; \textnormal{is a loop
based at } x\} \subseteq GL(n, \mathbb{R})
\end{equation}
\\
Furthermore one can easily see that $\textrm{Hol}_{x}(\nabla)$ has a group
structure and is independent of the basepoint $x \in M$ ``up to conjugation" if
$M$
is simply-connected: $P_{\gamma} \textrm{Hol}_{x}(\nabla) P^{-1}_{\gamma} = 
\textrm{Hol}_{y}(\nabla)$ for any piecewise smooth map $\gamma : [0,1] \to M$
with $\gamma(0) = x$ and $\gamma(1) = y$.   
\\  

If one considers $x \in M$ and $T_{x}M$ the tangent space at x, one can show
that the constant tensor $g$ i.e. $\nabla g = 0$ is preserved under the action
of
$\textrm{Hol}_{x}(\nabla)$ on $T_{x}M$ (acting on $g|_{x}$). Since $O(n)$ are
the group of transformations of $T_{x}M$ preserving  $g|_{x}$, we know that,
\begin{equation}
\textrm{Hol}_{x}(\nabla) \subseteq O(n)
\end{equation}
The holonomy group of a Riemannian Manifold is referred to as the Riemannian
Holonomy Group.
\\

If one instead considers only null-homotopic curves on our manifold (that is
curves that are homotopic to the constant curve) then one can suitably define
what is known as the Restricted Holonomy Group,
\begin{equation}
\textrm{Hol}^{0}_{x}(\nabla) = \{P_{\gamma} \; | \; \gamma \; \textnormal{is a
null-homotopic loop based at } x\} \subseteq GL(n, \mathbb{R})
\end{equation}
All the properties described above of the Holonomy group carry through. The two
notions of Holonomy are equivalent on simply-connected Manifolds.
\\

There is a classification theorem due to Marcel Berger [MB] that answers the
question
which subgroups of $O(n)$ can be the Holonomy group of some Riemannian Manifold
of dimension $n$. Berger found that for generic Riemannian Manifolds
$\textrm{Hol}(\nabla) = SO(n)$. Riemannian Metrics with $ \textrm{Hol}(\nabla)
\subseteq U(m)$ where $n = 2m$ are called K\"{a}hler Metrics. Riemannian Metrics
with $ \textrm{Hol}(\nabla) \subseteq SU(m)$ where $n = 2m$ are called
Calabi-Yau Metrics. Hence for our purpose we now have a mathematical definition
of a Calabi-Yau Manifold:
\begin{definition} (Calabi-Yau Manifold)
A Calabi-Yau Manifold is a compact K\"{a}hler Manifold $(M,J,g)$ of dimension $m
\geq 2$ with $\textrm{Hol}(\nabla) \subseteq SU(m)$.
\end{definition}
A consequence of $\textrm{Hol}(\nabla) \subseteq SU(m)$ is that the first Chern
Class of our Manifold vanishes. We now state a lemma which combined with the
proof of the Calabi-Conjecture will prove the existence of Calabi-Yau Manifolds,
\begin{lemma} Let $(M,J,g)$ be a K\"{a}hler Manifold. Then
${Hol}^{0}(\nabla) \subseteq SU(m) \iff $ g is Ricci-flat.
\end{lemma}

Assuming that the first Chern Class of our Manifold vanishes, by the Calabi
Conjecture we can find a unique metric in each K\"{a}hler Class with vanishing
Ricci Form. Hence $\textrm{Hol}(\nabla) \subseteq SU(m)$ and we can prove the
existence of Ricci-flat metrics.


\begin{thebibliography}{*****}

\bibitem[AF]{AF} A. Figalli, \emph{Lecture Notes in Partial Differential
Equations}, UT Austin PDE Course.

\bibitem[ANCG]{ANCG} A. Neitzke, \emph{Complex Geometry Lecture
Notes}, from Andy's personal webpage.

\bibitem[AM]{AM} A. Moroianu, \emph{Lectures on K\"{a}hler Geometry}, from
Andrei's personal webpage.

\bibitem[EG]{EG}E. Calabi,  \emph{On K\"{a}hler manifolds with vanishing
canonical
class,
Algebraic Geometry and
Topology, A symposium in honor of S. Lefschetz}, 1955, pp. 78-89.

\bibitem[DJ]{DJ}D.D. Joyce,  \emph{Compact Manifolds with Special Holonomy},
Oxford University Press, Oxford 2000

\bibitem[GT]{GT}D. Gilbarg, N.S. Trudinger \emph{Elliptic Partial Differential
Equations of Second Order}, Springer, Berlin, 1998.

\bibitem[MB]{MB} M. Berger, \emph{Sur les groupes d'holonomie des
variétés a connexion affine et des variétés riemanniennes}, Bull. Soc. Math.
France 83: 279–330, 1953.

\bibitem[POL]{POL}D. Polchinski, \emph{String Theory Volumes 1 and 2},
Cambridge Monographs on Mathematical Physics, Cambridge, 1998.

\bibitem[JST]{JST}J. J\"{o}st, \emph{Riemannian Geometry and Geometric
Analysis}, Springer, Berlin, 2002.

\bibitem[SS]{SS}S.S. Chern \emph{Characteristic classes of Hermitian
manifolds}, Ann. of
Math. 47, 1946, pp.
85-121.

\bibitem[SY]{SY}S.-T. Yau, \emph{On the Ricci Curvature of a Compact K\"{a}hler
Manifold and the
Complex Monge Amp\'{e}re Equation I
}, Communications on Pure and Applied Mathematics \textbf{31}, 1978, 339-411.




\end{thebibliography}
\end{document}